\documentclass[10pt]{article}
\usepackage{amsmath, amsthm, amsfonts, amssymb}
\usepackage[usenames,dvipsnames,svgnames,table]{xcolor}%e.g. {\color{red} ************* $****$}

\newtheorem{theorem}{Theorem}[section]
\newtheorem{corollary}[theorem]{Corollary}

\newtheorem{proposition}[theorem]{Proposition}

\newtheorem{remark}[theorem]{Remark}

\numberwithin{equation}{section}

\title{\bf Gradient bounds  and Liouville property for a class of hypoelliptic diffusion via coupling\thanks{Sponsored by the NSF of China (No. 11671076)}}
\author{Bin Qian\thanks{Department of Mathematics, Changshu Institute of Technology, Changshu, Jiangsu 215500.
 E-mail: binqiancn@126.com, binqiancn@gmail.com}\ \ \ \ Beibei Zhang\thanks{Department of Mathematics, Changshu Institute of Technology, Changshu, Jiangsu 215500. E-mail: zhangbb@whu.edu.cn}
 }\date{}

\begin{document}
\maketitle
\begin{abstract}
In this paper, we obtain the reverse Bakry-\'Emery type estimates for a  class of hypoelliptic diffusion operator by coupling method. The (right and reverse) Poincar\'e inequalities and the (right and reverse) logarithmic Sobolev inequalities are presented as  consequences of such estimates. Wang-Harnack inequality,  Hamilton's gradient estimate and Liouville property are also presented by reverse logarithmic Sobolev inequality.

\end{abstract}

\section{Introduction}
We consider the second order differential operator (Kolmogorov type operator)  \begin{equation}\label{op1}
\mathcal{L}=\frac12div(AD)+\langle x,BD\rangle,  x\in \mathbb{R}^{N}
\end{equation}
where $D=\left(\partial_{x_1},\cdots,\partial_{x_N}\right)$, $\mbox{div}$ and $\langle\cdot,\cdot\rangle$ denote, respectively, the gradient, the divergence and the inner product in $\mathbb{R}^N$.  The matrix $A=(a_{ij})$ and $B=(b_{ij})$ are $N\times N$ real constant matrices in the form
\begin{equation}\label{ABM}
A=\left(\begin{array}{cc}
A_0&0\\
0&0
\end{array}\right) \ \ \mbox{and}\ \
B=\left(
\begin{array}{ccccc}
0&B_1&0&\cdots&0\\
0&0&B_2&\cdots&0\\
\vdots&\vdots&\vdots&\ddots&\vdots\\
0&0&0&\cdots&B_r\\
0&0&0&\cdots&0
\end{array}
\right)
\end{equation}
 where  $A_0$ is  a  $m_0\times m_0$ symmetric and positive semidefinite and  $B_k$ is a $m_{k-1}\times m_k$ matrix of rank $m_k, k=1,2,\ldots,r$, with
$$
m_0\ge m_1\ge \cdots\ge m_r\ge1\ \ \ \mbox{and}\ \ \ \sum_{k=0}^rm_k=N.
$$

\eqref{op1} arises in the description of wide classes of stochastic processes and kinetic models and in mathematical finance, see \cite{LPP02} etc.
By Proposition 2.1 and Proposition 2.2 of \cite{LP94}. Hypothesis \eqref{ABM} implies that the operator $L$ verifies the classical H\"ormander condition \cite{H67}. Define, for every $t\in \mathbb{R}$, $E(t)=\exp(-tB^*)$ ($B^*$ is the transpose matrix of $B$) and
\begin{equation}\label{C-M}
\ \mathcal{C}(t):=\int_0^t E(s)AE^*(s)ds
\end{equation}
Under the assumption \eqref{ABM}, we have $\mathcal{C}(t)>0$ for every $t>0$, it means that $\mathcal{C}(t)>0$ is positive definite, see \cite{LP94}, \cite{LPP02} etc.  By H\"ormander's  theorem \cite{H67}, the operator $\mathcal{L}$ defined in \eqref{op1} is hypoelliptic and generators a Markov process $\textbf{X}_t$. The transition probability density of the processs $\textbf{X}_t$ is Gaussian with respect to the Lebesgue measure $dx$, see e.g {\cite[Page 148]{H67}},  {\cite[Theorem 1 and Theorem 4]{K72}}  or {\cite[Page 30]{LP94}} for explicit expression of the transition probability:
\begin{equation*}\label{fs}
h(x,t;\xi,\tau)=p(x-E(t-\tau)\xi,t-\tau)
\end{equation*}
$0\le \tau<t$, where $p(x,t;0,0)=0$ if $t\le0$ and
$$
p(x,t)=\frac{(2\pi)^{-\frac{N}{2}}}{\sqrt{\det \mathcal{C}(t)}}\exp\left(-\frac{1}{2}\langle \mathcal{C}^{-1}(t)x,x\rangle\right), \ t>0.
$$

We can express the Markov process $\textbf{X}_t=(X_1(t),\cdots, X_{r+1}(t))$ by the following stochastic differential equation
\begin{equation}\label{markov}\begin{cases}
dX_1(t)&=\sigma dB(t),\\
dX_{2}(t)&=B_1^*X_1(t)dt\\
&\vdots\\
dX_{r+1}(t)&=B_r^*X_r(t)dt
\end{cases}\end{equation}
where $\sigma$ is a $m_0\times m_0$ constant matrix satisfying $\sigma\sigma^*=A_0$, for $1\le i\le r$,  $B(t)$ is a Brownian motion in  $\mathbb{R}^{m_0}$ and $X_{i+1}(t)$ is a $\mathbb{R}^{m_i}-$valued process ($1\le i\le r$).

In particular, $r=1$ and $A_0=B_1=1$, \eqref{op1} reduces to the Kolmogorov operator initially introduced by A. N. Kolmogorov in \cite{Ko34},  which is the simplest example of a hypoelliptic seconde order differential operator, see \cite{H67}. In the case of $r=1$ and $A_0=B_1=id$, Baudoin et al \cite{BGM20} and the very recent paper \cite{BGHKM23} obtained gradient bounds and other functional inequalities including  reverse Poincar\'e inequality, reverse logarithmic Sobolev inequality by generalized $\Gamma-$calculus, see \cite{BE85,BGL14} for Bakry-\'Emery  $\Gamma-$calculus in Riemannian manifolds. The authors \cite{QZ1} generalized the results to the iterated Kolmogorov diffusion operator ($A_0=B_1=\cdots=B_r=id \ (r\ge1)$) by a similar  modified $\Gamma-$calculus. The main observation of \cite{QZ1} is the reverse Bakry-\'Emery type inequalities,  which generates the reverse Poincar\'e inequality and reverse logarithmic Sobolev inequality. For non-Gaussian degenerate  diffusions, Baudoin et el. \cite{BGM23} studied Wang-Harnack inequality (by reverse logarithmic Sobolev inequality) and the quasi-invariance property  for infinite dimensional Kolmogorov diffusions. See \cite{BBBC,BG17,BB12,GW12,Wa14,Ba17} etc. for functional inequalities and related topics for the hypoelliptic diffusion operators.

 On the other side, the Li-Yau gradient estimates for the positive solutions of \begin{equation}\label{heat}\mathcal{L}u=\partial_tu\end{equation} has been proved by Pascucci and Polidoro \cite{PP04} by the original variational argument due to Li and Yau \cite{LY86}, see also \cite{GL90}, \cite{LP94} and references therein. The Li-Yau type gradient estimate can give the parabolic Harnack inequality, the latter can compare the solution (heat) of \eqref{heat} of increasing time at two different  points. R. S. Hamilton \cite{Ham11} obtained the matrix differential Harnack estimate for hypoelliptic heat equation \eqref{heat} in the case of $r=1,A_0=B_1=1$ by using Riccati type equations. Huang \cite{HH14} extended Hamilton's result to slightly general model and conjectured that the matrix differential Harnack estimate holds for hypoelliptic operator \eqref{op1} with structure \eqref{ABM}.  This conjecture has been solved by the authors in \cite{QZ2}.  Meanwhile, there is a different gradient estimate,  discovered by R. S. Hamilton in \cite{Ham93}, for bounded positive solutions of $Lu=\partial_tu$ on compact Riemannian manifolds under  the lower bounded of the Ricci curvature. This type of gradient estimate (called Hamilton's elliptic gradient estimate) can give the Harnack inequality with power (see  for example  Corollary \ref{harn-ham3}), which can compare the heat of the same time at two different points.  One may ask whether Hamilton's elliptic gradient estimate hold for the hypoelliptic heat equation \eqref{heat}. Since the Ricci curvature (or Bakry-\'Emery $\Gamma_2$ curvature) of hypoelliptic operator like $\mathcal{L}$ is $-\infty$ every where,  the original method in \cite{Ham93} doesn't work. It is the start point  of this paper  to derive the Hamilton's elliptic gradient estimate for the hypoelliptic heat equation \eqref{heat}.

To this end, we intend to study the gradient bounds and funtional inequalities for the  hypoelliptic diffusion  operator \eqref{markov}. Since the lower bound of  the associated Bakry \'Emery $\Gamma_2$ curvature of $\mathcal{L}$ with structure \eqref{ABM} is unclear (see the computations of Kolmogorov diffusion operator and iterated Kolmogorov diffusion operator in \cite{BGM20,QZ1}), new method is needed. For special choices of $A$ and $B_1$, synchronous coupling was introduced in \cite[Proposition 2.10]{BGM20} to prove  the $L^q$ version of the right Bakry-\'Emery type estimate. It is said that it seems difficult to prove the reverse Poincar\'e and the reverse logarithmic Sobolev inequalities for the semigroup generated by Kolmogorov diffusion by using coupling techniques,  see Remark 2.11 in \cite{BGM20}.  Here we succeed in applying coupling techniques to derive the reverse Bakry-\'Emery type estimates, which imply  reverse Poincar\'e and the reverse logarithmic Sobolev inequalities, see Corollary \ref{BE2-cor}  or Corollary \ref{BE2-cor-3},  and Corollary \ref{BEln-cor-3}, which is the main contribution of this paper.

 This paper is organized as follows. To illustrate the coupling method we begin with the simple case $r=1$ in Section \ref{sec-2.1}. The Bakry-\'Emery type estimates with  free parameters are presented, see Proposition \ref{prop-BE}; Different choices of the free parameter gives right and reverse Bakry-\'Emery type estimates (Corollary \ref{BE2-cor}). Then we extend to the general case $r\ge1$ in section \ref{sec-2.2}, see Proposition \ref{prop-BE-3}, Proposition \ref{prop-BEln-3} and Corollary \ref{BE2-cor-3}, Corollary \ref{BEln-cor-3}. In section \ref{sec-3}, we give various applications of Bakry-Emery type inequalities, see Theorem \ref{PI-3} for the (right and reverse) Poincar\'e inequality, and Theorem \ref{thm-LSI-3} for the (right and reverse) logarithmic Sobolev inequality.  Wang-Harnack inequality,  Hamilton's elliptic gradient estimate and Liouville property are also presented as consequences of reverse logarithmic Sobolev inequality.

Throughout this paper, we use the following \\
{\bf Notations:} For any vector $x=(x_1,\cdots,x_d)^*,y=(y_1,\cdots,y_d)^*\in \mathbb{R}^{d}$, $q\ge 1$, we denote $\|x\|_q=\left(\sum_{i=1}^{d}|x_i|^q\right)^{\frac1q},\ \|x\|_{A,q}=\|\sigma x\|_q$ and the inner product $\langle x,y\rangle=\sum_{i=1}^dx_iy_i$. In short $\|x\|:=\|x\|_2=\sqrt{\langle x,x\rangle}$, $\|x\|_{A}:=\|x\|_{A,2}=\sqrt{x^*Ax}$.  We use the notation $\partial_jf=\partial_{x_j}f, 1\le j\le d$ for $f\in C^1\left(\mathbb{R}^d\right)$.  Let $I_i:=\sum_{k=1}^{i}m_{k-1},i=0,1,\cdots,r+1$, $\prod_{i=1}^kB_i:=B_1B_2\cdots B_k$, and by convention $\prod_{i=1}^0*=id$, $\sum_{k=1}^0*=0$. $P_t=e^{t\mathcal{L}}$ denotes the semigroup generated by the hypoelliptic diffusion operator $\mathcal{L}$, and $P_tf(x)=\int h(x,t;y,0)f(y)dy$.

For any  $f\in C^1(\mathbb{R}^{I_{r+1}})$, $x=\left(\begin{matrix}x^{(1)}\\ \cdots\\ x^{(r+1)}\end{matrix}\right)\in \mathbb{R}^{I_{r+1}}$, where $x^{(i)}=(x_{I_{i-1}+1},\cdots,x_{I_{i-1}+m_{i-1}})^*\in \mathbb{R}^{m_i}, i=1,2,\cdots,r+1$, we denote
$$\nabla^{(i)}f=\left(\partial_{{I_{i-1}+1}}f, \cdots,\partial_{ {I_{i-1}+m_{i-1}}}f\right)^* \ \mbox{and} \ \nabla^{}f=\left(\begin{matrix} &\nabla^{(1)}f\\&\cdots\\
&\nabla^{(r+1)}f\end{matrix}\right).$$

For  the second order differential operator $\mathcal{L}$ as in \eqref{op1}, we associate the {\it carr\'e du champ} $\Gamma$, which for smooth enough function $f,g$ is defined by
$$
\Gamma(f,g):=\frac12\left(\mathcal{L}(fg)-f\mathcal{L}g-g\mathcal{L}f\right), \ \Gamma(f):=\Gamma(f,f).
$$

\section{Bakry-\'Emery type inequality}\label{sec-2}

\subsection{Bakry-\'Emery type inequality: $r=1$.}\label{sec-2.1}
To illustrate the method, we begin with the case of $r=1$. Now the matrix $B$ in \eqref{ABM} has a simple form $\left(
\begin{array}{cc}
0&B_1\\
0&0
\end{array}
\right).
$   The associated Markov process ${\bf{X}}_t=(X_1(t),X_2(t))$ can be described as
 \begin{equation*}\label{markov2}\begin{cases}
dX_1(t)&=\sigma dB(t),\\
dX_{2}(t)&=B_1^*X_1(t)dt
\end{cases}\end{equation*}

 We have the following Bakry-\'Emery type estimate for the semigroup $P_tf(x)=\mathbb{E}(f({\bf X}_t))(x)$. \begin{proposition}\label{prop-BE}
Let $f\in C^2\left(\mathbb{R}^{m_0}\times\mathbb{R}^{m_1}\right)$ with bounded  second derivatives. Then for any fixed $t\ge0$, $x^{(1)}\in \mathbb{R}^{m_0}, x^{(2)}\in\mathbb{R}^{m_1}$ and any constant $\alpha_2\in \mathbb{R}$,
\begin{equation}\label{BE}
\left\|\nabla^{{(1)}}P_tf(x^{(1)},x^{(2)})+\alpha_2 B_1\nabla^{(2)}P_tf(x^{(1)},x^{(2)})\right\|_{A_0}^2\le P_t\left(\left\|\nabla^{(1)}f+(\alpha_2+t)B_1\nabla^{(2)}f\right\|_{A_0}^2\right)(x^{(1)},x^{(2)}).
\end{equation}
\end{proposition}\begin{proof}Consider two copies of hypoelliptic diffusions
\begin{equation}\label{coupling-1}\begin{split}
\mathbf{X}_t&=(X_1(t),X_2(t))=\left(x^{(1)}+\sigma B(t), x^{(2)}+tB_1^*x^{(1)}+\int_0^tB_1^*\sigma B(s)ds\right),\\
\widetilde{\mathbf{X}}_t&=(\widetilde{X}_1(t),\widetilde{X}_2(t))=\left(\widetilde{x}^{(1)}+\sigma \widetilde{B}(t), \widetilde{x}^{(2)}+tB_1^*\widetilde{x}^{(1)}+\int_0^tB_1^*\sigma \widetilde{B}(s)ds\right),\end{split}
\end{equation}
where   $B(t)$ and $\widetilde{B}(t)$ are two Brownian motions in $\mathbb{R}^{m_0}$ started at $0$. $\mathbf{X}_t$ starts at $(x^{(1)},x^{(2)})$ and $\widetilde{\mathbf{X}}_t$ starts at $(\widetilde{x}^{(1)},\widetilde{x}^{(2)})$. Now let us  synchronously couple $(B(t),\widetilde{B}(t))$ for all the time $t$ ($B(t)=\widetilde{B}(t),\forall t$) and take \begin{equation}\label{eq ini-1}
x^{(1)}=\widetilde{x}^{(1)}+\varepsilon v,\ x^{(2)}=\widetilde{x}^{(2)}+\alpha_2\varepsilon B^*_1v
\end{equation} for any real constant $\alpha_2\in \mathbb{R}$ and some constant vector $v\in\mathbb{R}^{m_0}$  so that, for any $t>0$,
\begin{equation}\label{eq ini-10}
{X}_1(t)-\widetilde{{X}}_1(t)=\varepsilon v,\ \ X_2(t)-\widetilde{X}_2(t)=(\alpha_2+t)\varepsilon B^*_1v.
\end{equation}
By using an estimate on the remainder $R$ of Taylor's approximation to $f$ and the assumption that  $f\in C^2(\mathbb{R}^{m_0}\times \mathbb{R}^{m_1})$ has bounded second derivatives,  there exists a positive constant $C_f\ge 0$,
\begin{align*}
|f(\mathbf{X}_t)-f(\widetilde{\mathbf{X}}_t)|&=\left|\nabla^{(1)}f(\widetilde{\mathbf{X}}_t)({X}_1(t)-\widetilde{{X}}_1(t))+\nabla^{(2)}f(\widetilde{\mathbf{X}}_t)({X}_2(t)-\widetilde{{X}}_2(t))+R\right|\\
&\le \varepsilon \left|\left(\nabla^{(1)}f(\widetilde{\mathbf{X}}_t)+(\alpha_2+t)B_1\nabla^{(2)}f(\widetilde{\mathbf{X}}_t) \right)\cdot v\right|+\varepsilon^2C_f\left(1+|\alpha_2|+t\right)^2\|B_1B_1^*\|_{HS}^2\|v\|^2,
\end{align*}
where $\|B_1B_1^*\|_{HS}$ is the Hilbert Schmidt norm of the $m_0\times m_0$ matrix $B_1B_1^*$. Using this estimate and Jensen's inequality, we have
\begin{align*}
&\left|P_tf(x^{(1)},x^{(2)})-P_tf(\widetilde{x}^{(1)},\widetilde{x}^{(2)})\right|\\
&=\left|\mathbb{E}^{((x^{(1)},x^{(2)}),(\widetilde{x}^{(1)},\widetilde{x}^{(2)}))}[f(\mathbf{X}_t)-f(\widetilde{\mathbf{X}}_t)]\right|\\
&\le \mathbb{E}^{((x^{(1)},x^{(2)}),(\widetilde{x}^{(1)},\widetilde{x}^{(2)}))}\left|f(\mathbf{X}_t)-f(\widetilde{\mathbf{X}}_t)\right|\\
&\le \varepsilon\mathbb{E}^{((x^{(1)},x^{(2)}),(\widetilde{x}^{(1)},\widetilde{x}^{(2)}))}\left|\left(\nabla^{(1)}f(\widetilde{\mathbf{X}}_t)+(\alpha_2+t)B_1\nabla^{(2)}f(\widetilde{\mathbf{X}}_t)\right)\cdot v\right|+\varepsilon^2C_f\left(1+|\alpha_2|+t\right)^2\|B_1B_1^*\|_{HS}^2\|v\|^2.
\end{align*}
Dividing out by $\varepsilon$ and taking $\varepsilon\to 0$ (hence $\widetilde{x}^{(1)}\to x^{(1)}$ and $\widetilde{x}^{(2)}\to x^{(2)}$ since $\widetilde{x}^{(2)}-x^{(2)}=\alpha_2\varepsilon B^*_1v$),  we have that
\begin{align*}
\limsup_{\varepsilon\to 0}\frac{\left|P_tf(x^{(1)},x^{(2)})-P_tf(\widetilde{x}^{(1)},\widetilde{x}^{(2)})\right|}{\varepsilon}\le P_t\left(\left|\left(\nabla^{(1)}f+(\alpha_2+t)B_1\nabla^{(2)}f\right)\cdot v\right|\right)(x^{(1)},x^{(2)}). \end{align*}
Meanwhile, we can see that
$$
\limsup_{\varepsilon\to 0}\frac{\left|P_tf(x^{(1)},x^{(2)})-P_tf(\widetilde{x}^{(1)},\widetilde{x}^{(2)})\right|}{\varepsilon} =\left|\left(\nabla^{(1)}P_tf(x^{(1)},x^{(2)})+\alpha_2 B_1\nabla^{(2)}P_tf(x^{(1)},x^{(2)})\right)\cdot v\right|$$
thus
$$
\left|\left(\nabla^{(1)}P_tf(x^{(1)},x^{(2)})+\alpha_2 B_1\nabla^{(2)}P_tf(x^{(1)},x^{(2)})\right)\cdot v\right|\le P_t\left|\left(\nabla^{(1)}f(x^{(1)},x^{(2)})+(\alpha_2+t)B_1\nabla^{(2)}f(x^{(1)},x^{(2)})\right)\cdot v\right|.$$
Notice that
\begin{align*}
&\left\|\nabla^{(1)}P_tf(x^{(1)},x^{(2)})+\alpha_2 B_1\nabla^{(2)}P_tf(x^{(1)},x^{(2)})\right\|_{A_0}\\
&=\left\|\sigma^*\left(\nabla^{(1)}P_tf(x^{(1)},x^{(2)})+\alpha_2 B_1\nabla^{(2)}P_tf(x^{(1)},x^{(2)})\right)\right\|\\
&=\sup_{\|v\|=1}\left|\left\langle\sigma^*\left(\nabla^{(1)}P_tf(x^{(1)},x^{(2)})+\alpha_2 B_1\nabla^{(2)}P_tf(x^{(1)},x^{(2)})\right),v\right\rangle\right|\\
&=\sup_{\|v\|=1}\left|\left\langle \left(\nabla^{(1)}P_tf(x^{(1)},x^{(2)})+\alpha_2 B_1\nabla^{(2)}P_tf(x^{(1)},x^{(2)})\right),\sigma v\right\rangle\right|\\
&\le \sup_{\|v\|=1} P_t\left(\left|\left\langle \left(\nabla^{(1)}f+(\alpha_2+t)B_1\nabla^{(2)}f\right),\sigma v\right\rangle\right|\right)(x^{(1)},x^{(2)})\\
&=\sup_{\|v\|=1} P_t\left(\left|\sigma^*\left\langle\left(\nabla^{(1)}f+(\alpha_2+t)B_1\nabla^{(2)}f\right), v\right\rangle\right|\right)(x^{(1)},x^{(2)}).
\end{align*}
By Cauchy-Schwarz inequality,  \begin{align*}
 P_t&\left(\left|\left\langle\sigma^*\left(\nabla^{(1)}f+(\alpha_2+t)B_1\nabla^{(2)}f\right), v\right\rangle\right|\right)(x^{(1)},x^{(2)})\\
 &\le \|v\|^{}\cdot P_t\left(\left\|\sigma^*\left(\nabla^{(1)}f+(\alpha_2+t)B_1\nabla^{(2)}f\right)\right\|^{}\right)(x^{(1)},x^{(2)})\\
 &\le\|v\|^{} \cdot P_t\left(\|\nabla^{(1)}f+(\alpha_2+t)B_1\nabla^{(2)}f\|_{A_0}^2\right)^{\frac12}(x^{(1)},x^{(2)}).
\end{align*}
Hence we have
$$
\left\|\nabla^{(1)}P_tf(x^{(1)},x^{(2)})+\alpha_2 B_1\nabla^{(2)}P_tf(x^{(1)},x^{(2)})\right\|_{A_0}^2\le  P_t\left(\|\nabla^{(1)}f+(\alpha_2+t)B_1\nabla^{(2)}f\|_{A_0}^2\right)(x^{(1)},x^{(2)}).
$$
Thus we complete the proof.
\end{proof}

\begin{remark}
In the above proposition,  we can assume $f$ is $C^1\left(\mathbb{R}^{m_0}\times\mathbb{R}^{m_1}\right)$ globally Lipschitzian, since $P_s$ has a Gaussian kernel.
\end{remark}

\begin{corollary}\label{BE2-cor}
Let $f\in C^1\left(\mathbb{R}^{m_0}\times\mathbb{R}^{m_1}\right)$ be a  globally Lipschitz function. Then for any $t\ge0$, $x^{(1)}\in \mathbb{R}^{m_0}, x^{(2)}\in\mathbb{R}^{m_1}$, we have
\begin{equation}\label{zero}
2\Gamma(P_tf)(x^{(1)},x^{(2)})\le P_t\left(\left\|\nabla^{(1)}f+tB_1\nabla^{(2)}f\right\|_{A_0}^2\right)(x^{(1)},x^{(2)}).
\end{equation}
and the following reverse inequality:
\begin{equation}\label{zero-t}
2P_t\left(\Gamma(f)\right)(x^{(1)},x^{(2)})\ge \left\|\nabla^{(1)}P_tf(x^{(1)},x^{(2)})-t B_1\nabla^{(2)}P_tf(x^{(1)},x^{(2)})\right\|_{A_0}^2.
\end{equation}
\end{corollary}
\begin{proof}
Notice that $\Gamma(f)(x^{(1)},x^{(2)})=\frac12\|\nabla^{(1)} f\|_{A_0}^2(x^{(1)},x^{(2)})$ for any $f\in C^1(\mathbb{R}^{m_0}\times \mathbb{R}^{m_1})$. If we take $\alpha_2=0$ in \eqref{BE}, then \eqref{zero} follows immediately. For the reverse inequality \eqref{zero-t}, we only need to take $\alpha_2=-t$.
\end{proof}

\begin{remark}\label{rem exp}\begin{description}
\item{(1).} Tracking back to \eqref{eq ini-1} and \eqref{eq ini-10}, the right Bakry-\'Emery type inequality \eqref{zero} follows from the coupling \eqref{coupling-1} with the starting point $x^{(1)}=\widetilde{x}^{(1)}+\varepsilon v$ and $x^{(2)}=\widetilde{x}^{(2)}$. While the reverse Bakry-\'Emery type inequality \eqref{zero-t} follows from the coupling \eqref{coupling-1} with  $X_2(t)=\widetilde{X}_2(t)$ for fixed $t>0$  by choosing  appropriate starting point. Hence the second coordinate $\cdot^{(2)}$ in $(\cdot^{(1)},\cdot^{(2)})$ plays important roles in the coupling \eqref{coupling-1}. While the choice of the difference between the first coordinates  is obvious, it is used  to obtain the gradient  of $P_tf$ along the vector $v$, see for example the Bismut type  formula  in \cite[Theorem 1.3.8]{Wa14}.
\item{(2).} After submitting  the manuscript, we are indicated by the reviewers that similar results have been obtained in the very recent paper \cite{BGHKM23} by a different approach for general $B$ in \eqref{ABM}. The Kolmogorov type operator \eqref{op1} with the nilpotent matrix $B$ in \eqref{ABM} is called principal part of the Kolmogorov type operator \eqref{op1} with the general matrix $B$, see \cite[(1.25),(1.26)]{LP94}. If the matrix $B$  in \eqref{op1} is not nilpotent, without loss of any generality, we can assume the  $(m_0+m_1)\times (m_0+m_1)$ matrix  $B$ has the following form
    $$\left(
    \begin{matrix}
    B_{11}&B_{12}\\
    B_{21}&B_{22}
    \end{matrix}\right)
    $$
where  $B_{11}$,  $B_{12}$,  $B_{21}$ and $B_{22}$  are  $m_0\times m_0$, $m_0\times m_1$, $ m_1\times m_0$,   $m_1\times m_1$ matrices respectively. The associated Markov process ${\bf{X}}_t=(X_1(t),X_2(t))$ can be described as
 \begin{equation*}\label{markov3}\begin{cases}
dX_1(t)&=\left(B_{11}^*X_1(t)+B_{21}^*X_2(t)\right)dt+\sigma dB(t),\\
dX_{2}(t)&=\left(B_{12}^*X_1(t)+B_{22}^*X_2(t)\right)dt.
\end{cases}\end{equation*}
In this case, $X_1(t)$ and $X_2(t)$ interact each other,    it seems difficult to obtain similar result as \eqref{BE} by similar procedure as the proof of Proposition \ref{BE}.
\end{description}
\end{remark}

\subsection{Bakry-\'Emery type inequality: $r\ge1$.}\label{sec-2.2}

In this subsection, we generalize the results for $r=1$ in above subsection to the case of $r\ge1$.
\begin{proposition}[Bakry-\'Emery type estimate]\label{prop-BE-3}
Let $f\in C^2\left(\mathbb{R}^{m_0}\times\cdots\times\mathbb{R}^{m_r}\right)$ with bounded  second derivatives. Then for any fixed $t\ge0$, $x=(x^{(1)},\cdots,x^{(r+1)})\in \mathbb{R}^{m_0}\times\cdots\times\mathbb{R}^{m_r}$ and any constants $\alpha_k\in \mathbb{R}\  (2\le k\le r+1)$,
\begin{equation}\label{BE-3}
\left\|\sum_{i'=1}^{r+1}\alpha_{i'}\prod_{j=1}^{i'-1}B_{j}\nabla^{(i')}P_tf(x)\right\|_{A_0}^2\le  P_t\left(\left\|\sum_{i'=1}^{r+1}\left(\sum_{i=0}^{i'-1}\alpha_{i'-i}\frac{t^i}{i!} \prod_{j=1}^{i'-1}B_{j}\right)\nabla^{(i')}f\right\|_{A_0}^2\right)(x),\end{equation}
where $\alpha_1=1$.
\end{proposition}
\begin{proof}Consider two copies of hypoelliptic diffusions
\begin{equation*}
\mathbf{X}_t=\left(\begin{matrix}
&X_1(t)\\
&X_2(t)\\
&\cdots\\
&X_{r+1}(t)
\end{matrix}\right)
, \ \ \widetilde{\mathbf{X}}_t=\left(\begin{matrix}
&\widetilde{X}_1(t)\\
&\widetilde{X}_2(t)\\
&\cdots\\
&\widetilde{X}_{r+1}(t)
\end{matrix}\right)
\end{equation*}
with $X_1(t)=x^{(1)}+\sigma B(t),\  \widetilde{X}_1(t)=\widetilde{x}^{(1)}+\sigma \widetilde{B}(t),$ and for $2\le k \le r+1$,
\begin{equation}\label{coupling-2}\begin{split}X_{k}(t)&=x^{(k)}+\sum_{i=1}^{k-1}\frac{t^i}{i!}\prod_{j=1}^iB_{k-j}^* x^{(k-i)}+\int_0^tdt_{k-1}\int_0^{t_{k-1}}dt_{k-2}\cdots\int_0^{t_2}\prod_{j=1}^{k-1}B_{k-j}^*\sigma B(t_1)dt_1,\\
\widetilde{X}_{k}(t)&=\widetilde{x}^{(k)}+\sum_{i=1}^{k-1}\frac{t^i}{i!}\prod_{j=1}^iB_{k-j}^* \widetilde{x}^{(k-i)}+\int_0^tdt_{k-1}\int_0^{t_{k-1}}dt_{k-2}\cdots\int_0^{t_2}\prod_{j=1}^{k-1}B_{k-j}^*\sigma \widetilde{B}(t_1)dt_1,
\end{split}\end{equation}
where $B(t)$ and $\widetilde{B}(t)$ are two Brownian motions in $\mathbb{R}^{m_0}$ started at $0$. $\mathbf{X}_t$ starts at $(x^{(1)},x^{(2)},\cdots,x^{(r+1)})$ and $\widetilde{\mathbf{X}}_t$ starts at $(\widetilde{x}^{(1)},\widetilde{x}^{(2)},\cdots,\widetilde{x}^{(r+1)})$. Now let us  synchronously couple $(B(t),\widetilde{B}(t))$ for all the time $t$ ($B(t)=\widetilde{B}(t),\forall t$) and take \begin{equation}\label{eq ini-2}
x^{(1)}=\widetilde{x}^{(1)}+ \varepsilon\alpha_1 v \mbox{ and } x^{(k)}=\widetilde{x}^{(k)}+\alpha_{k}\varepsilon \prod_{j=1}^{k-1}B_{k-j}^*v,2\le k\le r+1
\end{equation} for any constants $\alpha_j\in \mathbb{R},2\le j\le r+1$ with $\alpha_1=1$ and some constant vector $v\in\mathbb{R}^{m_0}$  so that, for any $t>0$, $2\le k\le r+1$,
$$
X_1(t)-\widetilde{X}_1(t)=\varepsilon \alpha_1 v,\ \ X_k(t)-\widetilde{X}_k(t)=\varepsilon\left(\sum_{i=0}^{k-1}\alpha_{k-i}\frac{t^i}{i!} \right)\prod_{j=1}^{k-1}B^*_{k-j}v
$$
Similar to the proof of Proposition \ref{prop-BE}, we have %%
%%By using an estimate on the remainder $R$ of Taylor's approximation to $f$ and the assumption that  $f\in C^2(\mathbb{R}^{m_0}\times \mathbb{R}^{m_1})$ has bounded second %%derivatives,  there exists a positive constant $C_f\ge 0$,
%%\begin{align*}
%%|f(\mathbf{X}_t)-f(\widetilde{\mathbf{X}}_t)|&=\left|\sum_{i'=1}^{r+1}\nabla^{(i')}f(\widetilde{\mathbf{X}}(t))\cdot(X_{i'}(t)-\widetilde{X}_{i'}(t))+R\right|\\
%%&\le \varepsilon \left|\sum_{i'=1}^{r+1}\left(\sum_{i=0}^{i'-1}\alpha_{i'-i}\frac{t^i}{i!} \prod_{j=1}^{i'-1}B_{j}\right)\nabla^{(i')}f(\widetilde{\mathbf{X}}(t))\cdot %%v\right|+\varepsilon^2C_fC_{B,\alpha,t}\|v\|^2,
%%\end{align*}
%%where $C_{B,\alpha,t}=\sum_{k=1}^{r+1}\left(\sum_{i=0}^{k-1}\alpha_{k-i}\frac{t^i}{i!}\right)^2\left\|\prod_{i=1}^{k-1}B_i\prod_{j=1}^{k-1}B^*_{k-j}\right\|_{HS}^2$ with the %%notation $\|\cdot\|_{HS}$ the Hilbert Schmidt norm of the $m_0\times m_0$ matrix $\cdot$. Using this estimate and Jensen's inequality, we have for %%$x=(x^{(1)},\cdots,x^{(r+1)}),\widetilde{x}=(\widetilde{x}^{(1)},\cdots,\widetilde{x}^{(r+1)})$,
\begin{align*}
&\left|P_tf(x)-P_tf(\widetilde{x})\right|%%\\
%%&=\left|\mathbb{E}^{(x,\widetilde{x})}[f(\mathbf{X}_t)-f(\widetilde{\mathbf{X}}_t)]\right|\\
%%&\le \mathbb{E}^{(x,\widetilde{x})}\left|f(\mathbf{X}_t)-f(\widetilde{\mathbf{X}}_t)\right|\\
\le \varepsilon\mathbb{E}^{(x,\widetilde{x})}\left|\sum_{i'=1}^{r+1}\left(\sum_{i=0}^{i'-1}\alpha_{i'-i}\frac{t^i}{i!} \prod_{j=1}^{i'-1}B_{j}\right)\nabla^{(i')}f(\widetilde{\mathbf{X}}_t)\cdot v\right|+\varepsilon^2C_fC_{B,\alpha,t}\|v\|^2
\end{align*}
where $C_{B,\alpha,t}=\sum_{k=1}^{r+1}\left(\sum_{i=0}^{k-1}\alpha_{k-i}\frac{t^i}{i!}\right)^2\left\|\prod_{i=1}^{k-1}B_i\prod_{j=1}^{k-1}B^*_{k-j}\right\|_{HS}^2$. Dividing out by $\varepsilon$ and taking $\varepsilon\to 0$ (hence $\widetilde{x}_k\to x_k$ for all $1\le k\le r+1$),  we have that
\begin{align*}
\limsup_{\varepsilon\to 0}\frac{\left|P_tf(x)-P_tf(\widetilde{x})\right|}{\varepsilon}\le P_t\left(\left|\sum_{i'=1}^{r+1}\left(\sum_{i=0}^{i'-1}\alpha_{i'-i}\frac{t^i}{i!} \prod_{j=1}^{i'-1}B_{j}\right)\nabla^{(i')}f\cdot v\right|\right)(x). \end{align*}
On the other hand, we have
$$
\limsup_{\varepsilon\to 0}\frac{\left|P_tf(x)-P_tf(\widetilde{x})\right|}{\varepsilon} =\left|\sum_{i'=1}^{r+1}\alpha_{i'}\prod_{j=1}^{i'-1}B_{j}\nabla^{(i')}P_tf(x)\cdot v \right|$$
thus
\begin{equation}\label{BEM-3}
\left|\sum_{i'=1}^{r+1}\alpha_{i'}\prod_{j=1}^{i'-1}B_{j}\nabla^{(i')}P_tf(x)\cdot v \right|\le P_t\left(\left|\sum_{i'=1}^{r+1}\left(\sum_{i=0}^{i'-1}\alpha_{i'-i}\frac{t^i}{i!} \prod_{j=1}^{i'-1}B_{j}\right)\nabla^{(i')}f\cdot v\right|\right)(x).\end{equation}
Hence we have
$$
\left\|\sum_{i'=1}^{r+1}\alpha_{i'}\prod_{j=1}^{i'-1}B_{j}\nabla^{(i')}P_tf(x)\right\|_{A_0}^2\le  P_t\left(\left\|\sum_{i'=1}^{r+1}\left(\sum_{i=0}^{i'-1}\alpha_{i'-i}\frac{t^i}{i!} \prod_{j=1}^{i'-1}B_{j}\right)\nabla^{(i')}f(\widetilde{\mathbf{X}}_t)\right\|_{A_0}^2\right)(x).
$$
Thus we complete the proof.
\end{proof}
\begin{remark}
\item{(1).}
 In the above proposition,  we can assume $f$ is $C^1\left(\mathbb{R}^{m_0}\times\cdots \times\mathbb{R}^{m_r}\right)$ globally Lipschitzian, since $P_s$ has a Gaussian kernel.
 \\
 \item{(2).} We also can obtain the $L^q(q\ge1)$ version of estimate \eqref{BE-3} as Proposition 2.10 in \cite{BGM20}, it follows from \eqref{BEM-3} and H\"older inequality.
\end{remark}

\begin{corollary}\label{BE2-cor-3}
Let $f\in C^1\left(\mathbb{R}^{m_0}\times\cdots\times\mathbb{R}^{m_r}\right)$ be a globally Lipschitz function. Then for any $t\ge0$, $x=(x^{(1)},\cdots,x^{(r+1)})\in \mathbb{R}^{m_0}\times\cdots\times\mathbb{R}^{m_r}$,  we have
\begin{equation}\label{zero-3}
2\Gamma(P_tf)(x)\le P_t\left(\left\langle E(-t)AE^*(-t)\nabla f,\nabla f\right\rangle\right)(x).
\end{equation}
and the following reverse inequality:
\begin{equation}\label{zero-t-3}
\left\langle E(t)AE^*(t)\nabla P_tf(x),\nabla P_tf(x)\right\rangle\le 2P_t\left(\Gamma(f)\right)(x).
\end{equation}
\end{corollary}
\begin{proof}
Notice that $\Gamma(f)=\frac12\|\nabla^{(1)} f\|_A$ for any $f\in C^1(\mathbb{R}^{m_0}\times\cdots\times\mathbb{R}^{m_r})$. If we take $\alpha_k=0$ $(2\le k\le r+1) $ in \eqref{BE-3}, then \begin{equation}\label{zero-3-1}
\Gamma(P_tf)(x)\le P_t\left(\left\|\sum_{k=0}^r\frac{t^k}{k!}\prod_{j=1}^kB_j\nabla^{(k+1)}f\right\|_{A_0}^2\right)(x).
\end{equation}
 Notice that
 \begin{align*}
 \left\|\sum_{k=0}^r\frac{t^k}{k!}\prod_{j=1}^kB_j\nabla^{(k+1)}f\right\|_{A_0}^2&=\sum_{k_1=0}^r\sum_{k_2=0}^r\left\langle {A_0}\frac{t^{k_1}}{k_1!}\prod_{j=1}^{k_1}B_j\nabla^{(k_1+1)}f,\frac{t^{k_2}}{k_2!}\prod_{j=1}^{k_2}B_j\nabla^{(k_2+1)}f\right\rangle\\
 &=\sum_{k_1=0}^r\sum_{k_2=0}^r\frac{t^{k_1+k_2}}{k_1!k_2!}\left\langle B_{k_2}^*\cdots B_1^*A_0B_1\cdots B_{k_1}\nabla^{(k_1+1)}f,\nabla^{(k_2+1)}f\right\rangle,
 \end{align*}
 and  meanwhile \begin{align*}
\left\langle E(-t)AE^*(-t)\nabla f,\nabla f\right\rangle&=\left\langle\sum_{k_1=0}^r\sum_{k_2=0}^r\left(E(-t)AE^*(-t)\right)_{k_2+1,k_1+1}\nabla^{(k_1+1)}f,\nabla^{(k_2+1)}f\right\rangle\\
&=\sum_{k_1=0}^r\sum_{k_2=0}^r\frac{t^{k_1+k_2}}{k_1!k_2!}\left\langle B_{k_2}^*\cdots B_1^*A_0B_1\cdots B_{k_1}\nabla^{(k_1+1)}f,\nabla^{(k_2+1)}f\right\rangle,
\end{align*}
where we apply the identity (3.6) in \cite{LP94} in the last equality and  $\left(E(-t)AE^*(-t)\right)_{k_2+1,k_1+1}$ means the block  at the position $(k_2+1,k_1+1)$ of the block  matrix $E(-t)AE^*(-t)$. It follows
\begin{equation}\label{iden}
\left\|\sum_{k=0}^r\frac{t^k}{k!}\prod_{j=1}^kB_j\nabla^{(k+1)}f\right\|_{A_0}^2=\left\langle E(-t)AE^*(-t)\nabla f,\nabla f\right\rangle.
\end{equation}
Combining with \eqref{zero-3-1} and \eqref{iden},  \eqref{zero-3} follows.

  For the reverse inequality \eqref{zero-t-3}, we only need to take
$$\alpha_k=(-1)^{k-1}\frac{t^{k-1}}{(k-1)!}, \ 2\le k\le r+1$$
which yields
$$\sum_{i=0}^{i'-1}\alpha_{i'-i}\frac{t^i}{i!}=0, 2\le i'\le r+1. $$
In this case, \eqref{BE-3} gives
$$
\left\|\sum_{i=1}^{r+1}\frac{(-1)^{i-1}t^{i-1}}{(i-1)!}\prod_{j=1}^{i-1}B_j\nabla^{(i)}P_tf(x)\right\|_{A_0}^2\le P_t\left(\Gamma(f)\right)(x).
$$
Together with \eqref{iden},  \eqref{zero-t-3} follows.
\end{proof}

\begin{remark}\label{rem exp2}
Tracking back to \eqref{eq ini-2}, \eqref{zero-3} follows from the coupling \eqref{coupling-2} with the starting point $x^{(1)}=\widetilde{x}^{(1)}+\varepsilon v$ and $x^{(k)}=\widetilde{x}^{(k)}$ for $2\le k\le r+1$. While \eqref{zero-t-3} follows from the coupling \eqref{coupling-2} with  $X_k(t)\equiv\widetilde{X}_k(t)$ for fixed $t$ and $2\le k\le r+1$   by choosing  appropriate starting point.

\end{remark}

Moreover we can obtain the following Bakry-\'Emery type inequality:

\begin{proposition}\label{prop-BEln-3}
Let $f\in C^2\left(\mathbb{R}^{m_0}\times\cdots\times\mathbb{R}^{m_r}\right)$ be a  positive, globally Lipschitz and bounded function.  Then for any fixed $t\ge0$, $x=(x^{(1)},\cdots,x^{(r+1)})\in \mathbb{R}^{m_0}\times\cdots\times\mathbb{R}^{m_r}$ and any constants $\alpha_k\in \mathbb{R} (2\le k\le r+1)$,
\begin{align}\label{BEln-3}
&P_tf(x)\left\|\sum_{i'=1}^{r+1}\alpha_{i'}\prod_{j=1}^{i'-1}B_{j}\nabla^{(i')}\ln P_tf(x)\right\|_{A_0}^2\le P_t\left(f\left\|\sum_{i'=1}^{r+1}\left(\sum_{i=0}^{i'-1}\alpha_{i'-i}\frac{t^i}{i!} \prod_{j=1}^{i'-1}B_{j}\right)\nabla^{(i')}\ln f\right\|^2_{A_0}\right)(x),
\end{align}
where $\alpha_1=1$.
\end{proposition}
\begin{proof}As before, we assume $f\in C^2\left(\mathbb{R}^{m_0}\times\cdots\times\mathbb{R}^{m_r}\right)$ is a  positive and bounded function with bounded first and second derivatives. Without loss of any generality, we can assume $f\ge \delta>0$, otherwise we consider $f+\delta$ instead, and then $\delta\to0$.
Let us consider the same hypoelliptic diffusions $\mathbf{X}_t, \widetilde{\mathbf{X}}_t$  in the proof of Proposition \ref{prop-BE-3} with the start points $x=(x^{(1)},\cdots,x^{(r+1)})$ and $\widetilde{x}=(\widetilde{x}^{(1)},\cdots,\widetilde{x}^{(r+1)})$ with $x^{(1)}=\widetilde{x}^{(1)}+ \varepsilon v $,  $x^{(k)}=\widetilde{x}^{(k)}+\alpha_{k}\varepsilon \prod_{j=1}^{k-1}B_{k-j}^*v$ for any $\alpha_k\in \mathbb{R} (2\le k\le r+1 )$ and $v\in \mathbb{R}^{m_0}$.

For any positive function $f\in C^2(\mathbb{R}^{m_0}\times \cdots\times \mathbb{R}^{m_r})$ with bounded first and second derivatives, there exists a positive constant $C_f\ge 0$,
\begin{equation*}\aligned
|\sqrt{f}(\mathbf{X}_t)-\sqrt{f}(\widetilde{\mathbf{X}}_t)|\le &\varepsilon \left|\sum_{i'=1}^{r+1}\left(\sum_{i=0}^{i'-1}\alpha_{i'-i}\frac{t^i}{i!} \prod_{j=1}^{i'-1}B_{j}\right)\nabla^{(i')}\sqrt{f}(\widetilde{\mathbf{X}}(t))\cdot v\right|+\varepsilon^2C_fC_{B,\alpha,t}\|v\|^2,
\endaligned\end{equation*}
where $\alpha_1=1$ and  $C_{B,\alpha,t}$ is defined in the proof of Proposition \ref{prop-BE-3}. It follows by Minkowski's inequality,
\begin{align*}%\label{eq-sqrt}%
&\left(\mathbb{E}^{(x,\widetilde{x})}\left|\sqrt{f}(\mathbf{X}_t)-\sqrt{f}(\widetilde{\mathbf{X}}_t)\right|^2\right)^{\frac12}\\
&\le \varepsilon\left(\mathbb{E}^{(x,\widetilde{x})}\left|\sum_{i'=1}^{r+1}\left(\sum_{i=0}^{i'-1}\alpha_{i'-i}\frac{t^i}{i!} \prod_{j=1}^{i'-1}B_{j}\right)\nabla^{(i')}\sqrt{f}(\widetilde{\mathbf{X}}(t))\cdot v\right|^2\right)^{\frac12}+\varepsilon^2C_fC_{B,\alpha,t}\|v\|^2\nonumber.
\end{align*}
Thus
\begin{eqnarray} \label{eq-sem-3}
&&\left|\sqrt{P_tf}(x)-\sqrt{P_tf}(\widetilde{x})\right|\left(\sqrt{P_tf}(x)+\sqrt{P_tf}(\widetilde{x})\right)\nonumber\\
&=&\left|P_tf(x)-P_tf(\widetilde{x})\right|\nonumber\\
&=&\left|\mathbb{E}^{(x,\widetilde{x})}[f(\mathbf{X}_t)-f(\widetilde{\mathbf{X}}_t)]\right|\nonumber\\
&\le& \mathbb{E}^{(x,\widetilde{x})}\left|f(\mathbf{X}_t)-f(\widetilde{\mathbf{X}}_t)\right|\nonumber\\
&= & \mathbb{E}^{(x,\widetilde{x})}\left|\left(\sqrt{f}(\mathbf{X}_t)-\sqrt{f}(\widetilde{\mathbf{X}}_t)\right)\left(\sqrt{f}(\mathbf{X}_t)+\sqrt{f}(\widetilde{\mathbf{X}}_t)\right)\right| \nonumber\\
&\le& \left(\mathbb{E}^{(x,\widetilde{x})}\left|\sqrt{f}(\mathbf{X}_t)-\sqrt{f}(\widetilde{\mathbf{X}}_t)\right|^2\right)^{\frac12}  \left(\mathbb{E}^{(x,\widetilde{x})}\left|\sqrt{f}(\mathbf{X}_t)+\sqrt{f}(\widetilde{\mathbf{X}}_t)\right|^2\right)^{\frac12}\nonumber\\
&\le& \left( \varepsilon\left(\mathbb{E}^{(x,\widetilde{x})}\left|\sum_{i'=1}^{r+1}\left(\sum_{i=0}^{i'-1}\alpha_{i'-i}\frac{t^i}{i!} \prod_{j=1}^{i'-1}B_{j}\right)\nabla^{(i')}\sqrt{f}(\widetilde{\mathbf{X}}(t))\cdot v\right|^2\right)^{\frac12}+\varepsilon^2C_fC_{B,\alpha,t}\|v\|^2\right)\nonumber\\
& & \cdot \left(\mathbb{E}^{(x,\widetilde{x})}\left|\sqrt{f}(\mathbf{X}_t)+\sqrt{f}(\widetilde{\mathbf{X}}_t)\right|^2\right)^{\frac12}.
\end{eqnarray}
where we use the H\"older inequality  in the second last inequality. It is easy to see that
\begin{align*}
\lim_{\varepsilon\to0}\mathbb{E}^{(x,\widetilde{x})}\left|\sqrt{f}(\mathbf{X}_t)+\sqrt{f}(\widetilde{\mathbf{X}}_t)\right|^2=4P_t\left(f\right)(x).
\end{align*}

Dividing out $\varepsilon \left(\sqrt{P_tf}(x)+\sqrt{P_tf}(\widetilde{x})\right)$ in the both side of \eqref{eq-sem-3} and letting $\varepsilon\to0$, it yields
\begin{align*}
&\left|\sum_{i'=1}^{r+1}\alpha_{i'}\prod_{j=1}^{i'-1}B_{j}\nabla^{(i')}\sqrt{P_tf}(x)\cdot v\right|\\
&=\limsup_{\varepsilon\to0}\frac{\left|\sqrt{P_tf}(x)-\sqrt{P_tf}(\widetilde{x})\right|}{\varepsilon}\\
&\le \left(P_t\left|\sum_{i'=1}^{r+1}\left(\sum_{i=0}^{i'-1}\alpha_{i'-i}\frac{t^i}{i!} \prod_{j=1}^{i'-1}B_{j}\right)\nabla^{(i')}\sqrt{f}\cdot v\right|^2\right)^{\frac12}(x).
\end{align*}
Similar to the proof of Proposition \ref{prop-BE}, we have
\begin{align*}
&\left\|\sum_{i'=1}^{r+1}\alpha_{i'}\prod_{j=1}^{i'-1}B_{j}\nabla^{(i')}\sqrt{P_tf}(x)\right\|_{A_0}\le \left(P_t\left\|\sum_{i'=1}^{r+1}\left(\sum_{i=0}^{i'-1}\alpha_{i'-i}\frac{t^i}{i!} \prod_{j=1}^{i'-1}B_{j}\right)\nabla^{(i')}\sqrt{f}\right\|^2_{A_0}\right)^{\frac12}(x).
\end{align*}
The desired result follows immediately.
\end{proof}

\begin{corollary}\label{BEln-cor-3}
Let $f\in C^2\left(\mathbb{R}^{m_0}\times\cdots\times\mathbb{R}^{m_r}\right)$ be a  positive, globally Lipschitz  and bounded function.  Then for any fixed $t\ge0$, $x=(x^{(1)},\cdots,x^{(r+1)})\in \mathbb{R}^{m_0}\times\cdots\times\mathbb{R}^{m_r}$,  we have
\begin{equation}\label{zero2-3}
2P_tf(x)\Gamma(\ln P_tf)(x)\le P_t\left(f\left\langle E(-t)AE^*(-t)\nabla \ln f,\nabla \ln f\right\rangle\right)(x).
\end{equation}
and
\begin{equation}\label{zero-t2-3}
P_tf(x)\left\langle E(t)AE^*(t)\nabla P_tf(x),\nabla P_tf(x)\right\rangle\le 2P_t\left(f\Gamma(\ln f)\right)(x).
\end{equation}
\end{corollary}
\begin{proof} The proof is similar to the one of Corollary \ref{BE2-cor-3}. Let $\alpha_k=0\  (2\le k\le r+1)$ in \eqref{BEln-3}, then
$$
2P_tf(x)\Gamma(\ln P_tf)(x)\le P_t\left(f\left\|\sum_{k=0}^r\frac{t^k}{k!}\prod_{j=1}^kB_j\nabla^{(k+1)}\ln f\right\|_{A_0}^2\right)(x).
$$
Together with \eqref{iden}, we have \eqref{zero2-3} follows immediately.
 For the reverse inequality \eqref{zero-t2-3}, we only need to take $\alpha_k=(-1)^{k-1}\frac{t^{k-1}}{(k-1)!}, \ 2\le k\le r+1$.
\end{proof}

\section{Applications to functional inequalities and Liouville property}\label{sec-3}

\begin{theorem}[Poincar\'e inequality and reverse Poincar\'e inequality]\label{PI-3}
Let $f\in C^1\left(\mathbb{R}^{m_0}\times\cdots\times\mathbb{R}^{m_r}\right)$ be a globally Lipschitz bounded function. Then for any $t\ge0$, $x=(x^{(1)},\cdots,x^{(r+1)})\in \mathbb{R}^{m_0}\times\cdots\times\mathbb{R}^{m_r}$,  we have
\begin{align}\label{PI-0-3}
P_t(f^2)-(P_tf)^2&\le P_t\left(\left\langle \int_0^tE(-(t-s))AE^*(-(t-s))ds\nabla f,\nabla f\right\rangle\right).
\end{align}
For $t>0$, the following reverse Poincar\'e inequality holds for any bounded function $f$:
\begin{align}\label{RPI-0-3}
P_t(f^2)-(P_tf)^2&\ge \left\langle \mathcal{C}(t)\nabla P_tf,\nabla P_tf\right\rangle.
\end{align}
Where $E(t)=\exp(-tB^*)$ and  $\mathcal{C}(t)$ is defined in \eqref{C-M}.
\end{theorem}
\begin{proof}
It follow from \eqref{zero-3},
\begin{align*}
P_t(f^2)-(P_tf)^2&=2\int_0^tP_s\Gamma(P_{t-s}f)ds\\
&\le \int_0^tP_t\left(\left\langle E(-(t-s))AE^*(-(t-s))\nabla f,\nabla f\right\rangle\right)ds\\
&=P_t\left(\left\langle \int_0^tE(-(t-s))AE^*(-(t-s))ds\nabla f,\nabla f\right\rangle\right).
\end{align*}
where the first identity follows from   \cite[Page 207]{BGL14}. Hence \eqref{PI-0-3} follows. Now we use the reverse Bakry-\'Emery type  inequality \eqref{zero-t-3},
\begin{align*}
P_t(f^2)-(P_tf)^2&\ge \int_0^t\left\langle E(s)AE^*(s)\nabla P_tf,\nabla P_tf\right\rangle ds\\
&=\left\langle\int_0^t E(s)AE^*(s)ds\nabla P_tf,\nabla P_tf\right\rangle\\
&=\left\langle \mathcal{C}(t)\nabla P_tf,\nabla P_tf\right\rangle.
\end{align*}
Hence we complete the proof.
\end{proof}

We can rewrite respectively  \eqref{PI-0-3} and \eqref{RPI-0-3}  as \begin{align*}P_t(f^2)-(P_tf)^2&\le
\sum_{k=0}^r\frac{t^{2k+1}}{(k!)^2(2k+1)}P_t\left(\left\|\prod_{j=1}^kB_j\nabla^{(k+1)}f\right\|_{A_0}\right)\\
&+2\sum_{0\le k_1<k_2\le r}\frac{t^{k_1+k_2+1}}{k_1!k_2!(k_1+k_2+1)}P_t\left(\left\langle A_0\prod_{j=1}^{k_1}B_j\nabla^{(k_1+1)}f,\prod_{j=1}^{k_2}B_j\nabla^{(k_2+1)}f\right\rangle\right).\nonumber
\end{align*}
and
\begin{align*}
P_t(f^2)-(P_tf)^2&\ge\sum_{k=0}^r\frac{t^{2k+1}}{(k!)^2(2k+1)}\left\|\prod_{j=1}^kB_j\nabla^{(k+1)}P_tf\right\|_{A_0} \\
&+2\sum_{0\le k_1<k_2\le r}\frac{(-1)^{k_1+k_2}t^{k_1+k_2+1}}{k_1!k_2!(k_1+k_2+1)}\left\langle A_0\prod_{j=1}^{k_1}B_j\nabla^{(k_1+1)}P_tf,\prod_{j=1}^{k_2}B_j\nabla^{(k_2+1)}P_tf\right\rangle.\nonumber
\end{align*}

\begin{theorem}[Logarithmic Sobolev inequality and reverse Logarithmic Sobolev inequality]\label{thm-LSI-3}
Let $f$ be a   positive, globally Lipschitz and bounded $C^2\left(\mathbb{R}^{m_0}\times\cdots\times\mathbb{R}^{m_r}\right)$ function. Then for any fixed $t\ge0$, $x=(x^{(1)},\cdots,x^{(r+1)})\in \mathbb{R}^{m_0}\times\cdots\times\mathbb{R}^{m_r}$,  we have
\begin{align}\label{LSI-0-3}
\hskip0ptP_t(f\ln f)-(P_tf)\ln P_tf&\le \frac12P_t\left(f\left\langle \int_0^tE(-(t-s))AE^*(-(t-s))ds\nabla \ln f,\nabla \ln f\right\rangle\right).
\end{align}
For $t>0$, the following reverse Logarithmic Sobolev inequality  holds for any positive bounded function $f$:
\begin{align}\label{RLSI-0-3}
P_t(f\ln f)-P_tf\ln P_tf&\ge \frac12P_tf\left\langle \mathcal{C}(t)\nabla\ln P_tf,\nabla \ln P_tf\right\rangle,
\end{align}
where $\mathcal{C}(t)$ defined in \eqref{C-M} is a positive matrix.
\end{theorem}
\begin{proof}
%%Having  Corollary \ref{BEln-cor-3} in hand,  the proof is similar to the one of Theorem \ref{thm-LSI}.
It follow from \eqref{zero2-3},
\begin{align*}
P_t(f\ln f)-(P_tf)\ln P_tf&=\int_0^tP_s\left(P_{t-s}f\Gamma(\ln P_{t-s}f)\right)ds\\
&\le \frac12\int_0^tP_t\left(f\left\langle E(-(t-s))AE^*(-(t-s))\nabla \ln f,\nabla \ln f\right\rangle\right)ds,
\end{align*}
where the first identity follows from \cite[Proposition 5.5.3]{BGL14}.  Hence \eqref{LSI-0-3} follows.

For the reverse Logarithmic Sobolev inequality \eqref{RLSI-0-3}, we only need to use the reverse inequality \eqref{zero-t2-3},
\begin{align*}
P_t(f\ln f)&-(P_tf)\ln P_tf=\int_0^tP_s\left(P_{t-s}f\Gamma(\ln P_{t-s}f)\right)ds\\
&\ge \frac12P_tf\int_0^t\langle E(t-s)AE^*(t-s)ds\nabla P_tf,\nabla P_tf\rangle ds.
\end{align*}
\eqref{RLSI-0-3} follows easily.
\end{proof}

\begin{theorem}\label{thm-Wang-3}
Let $f$ be a non-negative Borel bounded function on $\mathbb{R}^{m_0}\times\cdots\times\mathbb{R}^{m_r}$.
Then for every $t>0, x,y\in \mathbb{R}^{m_0}\times\cdots\times\mathbb{R}^{m_r} $ and $\alpha>1$, we have
$$
\left(P_tf\right)^{\alpha}(x)\le C_{\alpha}(t,x,y)P_tf^{\alpha}(y)
$$
with $C_{\alpha}=\exp\left(\frac{\alpha}{2(\alpha-1)}\left\langle \mathcal{C}^{-1}(t)(y-x),y-x\right\rangle\right)$,   $\mathcal{C}^{-1}(t)$ denotes the inverse matrix of the positive definite matrix $\mathcal{C}(t)$.

\end{theorem}
\begin{proof}
The proof is similar to the one of Theorem 3.11 in \cite{QZ1}. We prove it here for readers' convenience. Without loss of any generality, we can assume $f$ is non-negative and rapidly decreasing. Let $t$ be fixed. For any $x,y\in \mathbb{R}^{m_0 \times\cdots \times m_r }$, consider the curve $$\gamma(s)=x+s(y-x)$$ such that  $\gamma(0)=x,\gamma(1)=y$ and $\beta(s)=1+(\alpha-1)s, 0\le s\le 1$. Let
$$
\phi(s):=\frac{\alpha}{\beta(s)}\ln P_tf^{\beta(s)}(\gamma(s)), \  0\le s\le 1.
$$
  Differentiating with respect to $s$  and using \eqref{RLSI-0-3} yields
\begin{align*}
\phi'(s)&=\frac{\alpha(\alpha-1)}{\beta^2(s)}\frac{P_t(f^{\beta(s)}\ln f^{\beta(s)})-(P_tf^{\beta(s)})\ln P_tf^{\beta(s)} }{P_tf^{\beta(s)}}-\frac{\alpha}{\beta(s)}\langle\nabla \ln P_tf^{\beta(s)},y-x\rangle\\
&\ge \frac{\alpha(\alpha-1)}{2\beta^2(s)}\langle \mathcal{C}(t)\nabla\ln P_tf^{\beta(s)},\nabla\ln P_tf^{\beta(s)}\rangle-\frac{\alpha}{\beta(s)}\sqrt{\langle \mathcal{C}(t) \nabla\ln P_tf^{\beta(s)},\nabla\ln P_tf^{\beta(s)}\rangle}\sqrt{\langle \mathcal{C}^{-1}(t)(y-x),y-x\rangle}\\
&\ge -\frac{\alpha}{2(\alpha-1)}\langle \mathcal{C}^{-1}(t)(y-x),y-x\rangle,
\end{align*}
where we use  the inequality $\langle x,y\rangle\le \sqrt{\langle \mathcal{C}(t)x,x\rangle}\sqrt{\langle \mathcal{C}^{-1}(t)y,y\rangle}$ for any $x,y\in\mathbb{R}^{m_0+\cdots+m_r}$ and for any positive matrix $\mathcal{C}(t)$ in the second last inequality and use the elementary inequality $ax^2-bx\ge -\frac{b^2}{4a}$ for $a>0$ in the last one.

Integrating $s$ from $0$ to $1$ yields
$$
\ln P_tf^{\alpha}(y)-\ln (P_tf)^{\alpha}(x)\ge -\frac{\alpha}{2(\alpha-1)}\langle \mathcal{C}^{-1}(t)(y-x),y-x\rangle.
$$
The desired result follows.
\end{proof}

We also have the following Hamilton's elliptic gradient estimate by the reverse logarithmic Sobolev inequality.
\begin{proposition}\label{Ham3}
Let $u$ be the positive  bounded solution to the hypoelliptic heat equation $\mathcal{L} u(x,t)=\partial_tu(x,t)$ with the initial value  $0\le f\le C$ ($C$ is a constant), we have for any  $x\in\mathbb{R}^{m_0}\times\cdots \times\mathbb{R}^{m_r}$ and $t>0$, \begin{equation}\label{Ham-eq3}
\frac12\langle \mathcal{C}(t)\nabla \ln u,\nabla \ln u\rangle\le \ln \frac{C}{u}.
\end{equation}

\end{proposition}
\begin{proof} It is known that $u(x,t)=P_tf(x)$. Then by the reverse logarithmic Sobolev inequality \eqref{RLSI-0-3},
$$
\frac12u\langle \mathcal{C}(t)\nabla \ln u,\nabla \ln u\rangle\le  u\ln(C)-u\ln u.
$$
The desired result follows.
\end{proof}

\begin{corollary}\label{harn-ham3}
Under the assumption of Proposition \ref{Ham3}, we have for $t>0$, $\alpha>1$ and any  $x,y\in\mathbb{R}^{m_0}\times\cdots\times \mathbb{R}^{m_r}$,
$$
u^{\alpha}(x,t)\le u(y,t)C^{\alpha-1}e^{\frac{\alpha}{2(\alpha-1)}\langle \mathcal{C}^{-1}(t) (y-x),y-x \rangle}.
$$

\end{corollary}
\begin{proof}
Let $t$ be fixed. For any $x,y\in \mathbb{R}^{m_0+\cdots+ m_r}$, consider the curve $$\gamma(s)=x+s(y-x), \ 0\le s\le 1$$ such that  $\gamma(0)=x,\gamma(1)=y$.
We have \begin{align*}
\sqrt{\ln\frac{C}{u}(y,t)}-\sqrt{\ln\frac{C}{u}(x,t)}&=\int_0^1\frac{d}{ds}\sqrt{\ln\frac{C}{u}(\gamma(s),t)}ds\\
&=\int_0^1\frac{\nabla \ln u(\gamma(s),t)\cdot\dot{\gamma}(s)}{2\sqrt{\ln\frac{C}{u}(\gamma(s),t)}}ds\\
&\le \int_0^1 \frac{\sqrt{\langle \mathcal{C}(t) \nabla \ln u,\nabla \ln u\rangle}\sqrt{\langle \mathcal{C}^{-1}(t) (y-x),y-x \rangle}}{2\sqrt{\ln\frac{C}{u}(\gamma(s),t)}} ds\\
&\le \frac{1}{\sqrt{2}} \sqrt{\langle \mathcal{C}^{-1}(t) (y-x),y-x \rangle},
\end{align*}
where we apply the Hamilton's gradient estimate \eqref{Ham-eq3} in the last inequality and in the second last inequality we apply the elementary inequality $\langle v_1,v_2\rangle\le\langle \mathcal{C}(t)v_1,v_1\rangle^{\frac12}\langle \mathcal{C}^{-1}(t)v_2,v_2\rangle^{\frac12}$  for any  $v_1,v_2\in\mathbb{R}^{m_0+\cdots+m_r}$ and the positive definite matrix  $\mathcal{C}(t)$.
Thus for any $\alpha>1$, applying the elementary inequality $(a+b)^2\le \alpha a^2+\frac{\alpha}{\alpha-1}b^2$ for $a,b\in \mathbb{R}$,
\begin{align*}
\ln\frac{C}{u}(y,t)&\le \left(\sqrt{\ln\frac{C}{u}(x,t)}+\frac{\sqrt{2}}2 \sqrt{\langle \mathcal{C}^{-1} (y-x),y-x \rangle}\right)^2\\
&\le \alpha \ln\frac{C}{u(x,t)}+\frac{\alpha}{2(\alpha-1)}\langle \mathcal{C}^{-1} (y-x),y-x \rangle.
\end{align*}
The desired result follows.
\end{proof}

As a consequence of Proposition \ref{Ham3}, the following Liouville property holds.
\begin{theorem}\label{Liou3}
Let $u$ be the positive  bounded solution to the hypoelliptic Laplacian equation $\mathcal{L} u(x)=0$.  Then $u$ must be a constant.
\end{theorem}
\begin{proof}
Assume $C$ is the upper bound of $u$, we have $u<C+1$.  Define the dilations group on $\mathbb{R}^N$ by
$$
\delta_{\lambda}=\text{diag}\left(\lambda I_{m_0},\lambda^3I_{m_1},\cdots,\lambda^{2r+1}I_{m_r}\right), \lambda>0
$$
where $I_{m_k}$ is the $m_k\times m_k$ identity matrix.  We can express $\mathcal{C}(t)$ as
\begin{equation}\label{scaling}
\mathcal{C}(t)=\delta_{\sqrt{t}}\mathcal{C}(1)\delta_{\sqrt{t}},
\end{equation}
see Proposition 2.3 in \cite{LP94}. Then we have by \eqref{Ham-eq3}

\begin{align*}
 \ln \frac{C+1}{u}&\ge \frac12\langle \mathcal{C}(1) \delta_{\sqrt{t}}\nabla \ln u,\delta_{\sqrt{t}}\nabla \ln u\rangle\\
 &\ge\frac12 \lambda_{\text{min}}\frac{1}{u^2}\left\|\delta_{\sqrt{t}}\nabla u\right\|^2,
\end{align*}
where $\lambda_{\text{min}}$ is the minimal eigenvalue of the positive matrix $\mathcal{C}(1)$.  Since $u$ is independent of $t$, letting $t\to \infty$ in the above inequality, we obtain $\nabla u\equiv0$. Hence $u$ must be a constant.
\end{proof}
\begin{remark}
The above Liouville property can \textbf{not} follow by the Li-Yau type inequality (see Proposition 4.2 in \cite{PP04}).
\end{remark}

\section*{ Acknowledgement}
The authors would like to express sincere thanks to the anonymous referee for his/her value suggestion and comment which
     greatly improve the quality of the paper.

\end{document}